\documentclass[11pt]{article}
\usepackage{amsthm}
\usepackage{latexsym}
\usepackage{amsmath}
\usepackage{amssymb}
\usepackage{epsfig}
\usepackage[matrix,arrow,cmtip,curve,dvips]{xy}

\newtheorem{theorem}{Theorem}[section]                                                                                                                                                                                                                                                                                                                                                                                        
\newtheorem{lemma}[theorem]{Lemma}
\newtheorem{corollary}[theorem]{Corollary}
\newtheorem{proposition}[theorem]{Proposition}
\newtheorem{definition}[theorem]{Definition}

\newcommand{\reg}{\mbox{\upshape ref\,}}
\newcommand{\pdim}{\mbox{\upshape pdim\,}}
\newcommand{\depth}{\mbox{\upshape depth\,}}
\newcommand{\del}{\mbox{\upshape del}}
\newcommand{\lk}{\mbox{\upshape lk}}

\begin{document}

\title{Algebraic properties of edge ideals\\
 via combinatorial topology}

\author{Anton Dochtermann and Alexander Engstr\"{o}m\\
\hspace{.3 in}\\
\textit{Dedicated to Anders Bj\"{o}rner on the occasion of his 60th birthday.}}

\date\today

\maketitle

\abstract{We apply some basic notions from combinatorial topology to establish various algebraic properties of edge ideals of graphs and more general Stanley-Reisner rings.  In this way we provide new short proofs of some theorems from the literature regarding linearity, Betti numbers, and (sequentially) Cohen-Macaulay properties of edges ideals associated to chordal, complements of chordal, and Ferrers graphs, as well as trees and forests.  Our approach unifies (and in many cases strengthens) these results and also provides combinatorial/enumerative interpretations of certain algebraic properties.  We apply our setup to obtain new results regarding algebraic properties of edge ideals in the context of local changes to a graph (adding whiskers and ears) as well as bounded vertex degree.  These methods also lead to recursive relations among certain generating functions of Betti numbers which we use to establish new formulas for the projective dimension of edge ideals.  We use only well-known tools from combinatorial topology along the lines of independence complexes of graphs, (not necessarily pure) vertex decomposability, shellability, etc.}

\section{Introduction}
Suppose $G$ is a finite simple graph with vertex set $[n] = \{1, \dots, n\}$ and edge set $E(G)$, and let $S := k[x_1, \dots, x_n]$ denote the polynomial ring on $n$ variables over some field $k$.  We define the \emph{edge ideal} $I_G \subseteq S$ to be the ideal generated by all monomials $x_i x_j$ whenever $ij \in E(G)$.  The natural problem is to then obtain information regarding the algebraic invariants of the $S$-module $R_G := S/I_G$ in terms of the combinatorial data provided by the graph $G$.  The study of edge ideals of graphs has become popular recently, and many papers have been written addressing various algebraic properties of edge ideals associated to various classes of graphs.  These results occupy many journal pages and often involve complicated (mostly `algebraic') arguments which seem to disregard the underlying connections to other branches of mathematics.  The proofs are often specifically crafted to address a particular graph class or algebraic property and hence do not generalize well to study other situations.

The main goal of this paper is to illustrate how one can use standard techniques from combinatorial topology (in the spirit of \cite{B}) to study algebraic properties of edge ideals.  In this way we recover and extend well-known results (often with very short and simple proofs) and at the same time provide new answers to open questions posed in previous papers.  Our methods give a unified approach to the study of various properties of edge ideals employing only elementary topological and combinatorial methods.  It is our hope that these methods will find further applications to the study of edge ideals.

For us the topological machinery will enter the picture when we view edge ideals as a special case of the more general theory of Stanley-Reisner ideals (and rings).  In this context one begins with a simplicial complex $\Delta$ on the vertices $\{1, \dots, n\}$ and associates to it the Stanley-Reisner ideal $I_\Delta$ generated by monomials corresponding to nonfaces of $\Delta$; the Stanley-Reisner ring is then the quotient $R_\Delta := S/I_\Delta$.  Stanley-Reisner ideals are precisely the square-free monomial ideals of $S$.  Edge ideals are the special case that $I_\Delta$ is generated in degree 2, and we can recover $\Delta$ as ${\tt Ind}(G)$, the independence complex of the graph $G$ (or equivalently as ${\tt Cl}(\bar G)$, the clique complex of the complement of $G$).  In the case of Stanley-Reisner rings, there is a strong (and well-known) connection between the topology of $\Delta$ and certain algebraic invariants of the ring $R_\Delta$.   Perhaps the most well-known such result is Hochster's formula from \cite{H1} (Theorem \ref{hochster} below), which gives an explicit formula for the Betti numbers of the Stanley-Reisner ring in terms of the topology of induced subcomplexes of $\Delta$.  

Many of our methods and results will involve combining the `right' combinatorial topological notions with basic methods for understanding their topology.  For the most part the classes of complexes that we consider will be those defined in a recursive manner, as these are particularly well suited to applications of tools such as Hochster's formula.  These include (not necessarily pure) shellable, vertex-decomposable, and dismantlable complexes (see the next section for definitions).  In the context of topological combinatorics these are popular and well-studied classes of complexes, and here we see an interesting connection to the algebraic study of Stanley-Reisner ideals.

The rest of the paper is organized as follows.  In section 2 we review some basic notions from combinatorial topology and the theory of resolutions of ideals.  In section 3 we discuss the case of edge ideals of graphs $G$ where $G$ is the \emph{complement} of a chordal graph.  Here we are able to give a simple proof of Fr\"{o}berg's main theorem from \cite{F}.
\begin{theorem}(Theorem \ref{Froberg})
For any graph $G$ the edge ideal $I_G$ has a linear resolution if and only if $G$ is the complement of a chordal graph.
\end{theorem}
\noindent In addition, our short proof gives a combinatorial interpretation of the Betti numbers of the complements of chordal graphs.  

In the case that $G$ is the complement of a chordal graph and is also \emph{bipartite} it can be shown that $G$ is a so-called \emph{Ferrers graph} (a bipartite graph associated to a given Ferrers diagram).  We are able to recover a formula for the Betti numbers of edge ideals Ferrers graphs, a result first established by Corso and Nagel in \cite{CN}.  Our proof is combinatorial in nature and provides the following enumerative interpretation for the Betti numbers of such graphs, answering a question posed in \cite{CN}.
\begin{theorem}(Theorem \ref{Ferrerbetti})
If $G_{\lambda}$ is a Ferrers graph associated to the partition $\lambda = (\lambda_1 \geq \cdots \geq \lambda_n)$, then the Betti numbers of $G_{\lambda}$ are zero unless $j = i+1$, in which case $\beta_{i,i+1}(G_{\lambda})$ is the number of rectangles of size $i+1$ in $\lambda$.  This number is given explictly by:
\[\beta_{i,i+1}(G_{\lambda}) = {\lambda_1 \choose i} + {\lambda_2+1 \choose i} + {\lambda_3+2 \choose i} + \cdots + {\lambda_n + n-1 \choose i} - {n \choose i+1}.\]
\end{theorem}

In section 4 we discuss the case of edge ideals of graphs $G$ in the case that $G$ is a chordal graph.  Here we provide a short proof of the following theorem, a strengthening of the main result of Francisco and Van Tuyl from \cite{FT} and a related result of Van Tuyl and Villareal from \cite{TV}. 
\begin{theorem}(Theorem \ref{chordalVD})
If $G$ is a chordal graph then the complex ${\tt Ind}(G)$ is vertex-decomposable and hence the ideal $I_G$ is sequentially Cohen-Macaulay.
\end{theorem}
Vertex-decomposable complexes are shellable and since interval graphs are chordal, this theorem also extends the main result of Billera and Myers from \cite{BM}, where it is shown that the order complex of a finite interval order is shellable.  In this section we also answer in the affirmative a suggestion/conjecture made in \cite{FT} regarding the sequentially Cohen-Macaulay property of cycles with an appended triangle (an operation which we call `adding an \emph{ear}').

\begin{proposition}(Proposition \ref{earcycle})
For $r \geq 3$, let $\tilde C_r$ be the graph obtained by adding an ear to an $r$-cycle.  Then the ideal $I_{\tilde C_r}$ is Cohen-Macaulay.
\end{proposition}

This idea of making small changes to a graph to obtain (sequentially) Cohen-Macaulay graph ideals seems to be of some interest to algebraists, and is also explored in \cite{V} and \cite{FH}.  In these papers, the authors introduce the notion of adding a \emph{whisker} of a graph $G$ at a vertex $v \in G$, which is by definition the addition of a new vertex $v^\prime$ and a new edge $(v,v^\prime)$.  Although our methods do not seem to recover results from \cite{FH} regarding sequentially Cohen-Macaulay graphs, we are able to give a short proof of the following result, a strengthening of a theorem of Villareal from \cite{V}.

\begin{theorem}(Theorem \ref{whiskerVD})
Let $G$ be a graph and let $G^\prime$ be the graph obtained by adding whiskers to every vertex $v \in G$.  Then the complex ${\tt Ind}(G^\prime)$ is pure and vertex-decomposable and hence the ideal $I_{G^\prime}$ is Cohen-Macaulay.
\end{theorem}

In section 5 we use basic notions from combinatorial topology to obtain bounds on the projective dimension of edge ideals for certain classes of graphs; one can view this as  a strengthening of the Hilbert syzygy theorem for resolutions of such ideals.  For several classes of graphs the connectivity of the associated independence complexes can be bounded from below by $an+b$ where $n$ is the number of vertices and $a$ and $b$ are fixed constants for that class. We show that the projective dimension of the edge ideal of a graph with $n$ vertices from such a class is at most $n(1-a)-b-1$. One result along these lines is the following.

\begin{proposition}(Corollary \ref{enkeltLemma})
If $G$ is a graph on $n$ vertices with maximal degree $d\geq 1$ then
the projective dimension of $R_G$ is at most $n\left(1-\frac{1}{2d}\right)+\frac{1}{2d}$.
\end{proposition}

In section 6 we introduce a generating function $\mathcal{B}(G;x,y)=\sum_{i,j} \beta_{i,j}(G)x^{j-i}y^i$ for the Betti numbers and use simple tools from combinatorial topology to derive certain relations for edge ideals of graphs.   We use these relations to show that the Betti numbers for a large
class of graphs is independent of the ground field, and to also provide new recursive formulas for projective dimension and regularity of $I_G$ in the case that $G$ is a forest.

\section{Background}
In this section we review some basic facts and constructions from the combinatorial topology of simplicial complexes and also review some related tools from the study of Stanley-Reisner rings.  

\subsection{Combinatorial topology}
The topological spaces most relevant to our study are (geometric realizations of) simplicial complexes.  A simplicial complex $\Delta$ is by definition a collection of subsets of some ground set $\Delta^0$ (called the \emph{vertices} of $\Delta$ and usually taken to be the set $[n] = \{1, \dots, n\}$) which are closed under taking subsets.  For us a \emph{facet} of a simplicial complex is an inclusion maximal face, and the simplicial complex $\Delta$ is called \emph{pure} if all the facets are of the same dimension.  If $\sigma \in \Delta$ is a face of a simplicial complex $\Delta$, the \emph{deletion} and \emph{link} of $\sigma$ are defined according to
\[\del_{\Delta}(\sigma) := \{\tau \in \Delta: \tau \cap \sigma = \emptyset \},\]
\[\lk_{\Delta}(\sigma) := \{\tau \in \Delta: \tau \cap \sigma = \emptyset, \tau \cup \sigma \in \Delta \}. \]

We next identify certain classes of simplicial complexes which arise in the context of edge ideals of graphs.  We take the first definition from \cite{J}.  
\begin{definition}\label{defVD}
A (not necessarily pure) simplicial complex $\Delta$ is \emph{vertex-decomposable} if either 
\begin{enumerate}
\item
$\Delta$ is a simplex, or
\item
$\Delta$ contains a vertex $v$ such that $\del_{\Delta}(v)$ and $\lk_{\Delta}(v)$ are vertex-decomposable, and such that every facet of $\del_{\Delta}(v)$ is a facet of $\Delta$.
\end{enumerate}
\end{definition}

\noindent
A related notion is that of \emph{non-pure} shellability, first introduced by Bj\"{o}rner and Wachs in \cite{BW}.

\begin{definition}
A (not necessarily pure) simplicial complex $\Delta$ is \emph{shellable} if its facets can be arranged in a linear order $F_1, F_2, \dots, F_t$ such that the subcomplex $\big( \displaystyle\bigcup_{i=1}^{k-1} \bar F_i \big) \cap \bar F_k$ is pure and $(\dim F_k - 1)$-dimensional, for all $2 \leq k \leq t$.
\end{definition}

\noindent
Note that when the complex $\Delta$ is pure, this definition recovers the more classical notion from \cite{Z}.  

One can also give a combinatorial characterization of a sequentially Cohen-Macaulay simplicial complex, as discussed in \cite{BWW}.  For a simplicial complex $\Delta$ and for $0 \leq m \leq \dim \Delta$, we let $\Delta^{<m>}$ denote the subcomplex of $\Delta$ generated by its facets of dimension at least $m$.

\begin{definition}
A simplicial complex $\Delta$ is \emph{sequentially acyclic} (over $k$) if $\tilde H_r(\Delta^{<m>};k) = 0$ for all $r < m \leq \dim \Delta$.

A simplicial complex $\Delta$ is \emph{sequentially Cohen-Macaulay (CM)} over $k$ if $\lk_{\Delta}(F)$ is sequentially acyclic over $k$ for all $F \in \Delta$.
\end{definition}

It has been shown (see for example \cite{BWW}) that a complex $\Delta$ is sequentially CM if and only if the associated Stanley-Reisner ring is sequentially CM in the algebraic sense; we refer to Section 4 for a definition of the latter.

One can check (see \cite{J} or \cite{B}) that for any field $k$ the following (strict) implications hold:

\begin{center}
Vertex-decomposable $\Rightarrow$ shellable $\Rightarrow$ sequentially CM over ${\mathbb Z}$ $\Rightarrow$ sequentially CM over $k$.
\end{center}

There are several simplicial complexes that one can assign to a given graph $G$.  The \emph{independence complex} ${\tt Ind}(G)$ is the simplicial complex on the vertices of $G$, with faces given by collections of vertices which do no contain an edge from $G$.  The \emph{clique complex} ${\tt Cl}(G)$ is the simplicial complex on the \emph{looped} vertices of $G$ whose faces are given by collections of vertices which form a clique (complete subgraph) in $G$.  These notions are of course related in the sense that ${\tt Ind}(G) = {\tt Cl}(\bar G)$, where $\bar G$ is the \emph{complement} of $G$.  In understanding the topology of independence complexes, we will make use of the following fact from \cite{E1}.

\begin{lemma} \label{indcomplex}
For any graph $G$ we have isomorphisms of simplicial complexes:
\[\del_{{\tt Ind}(G)}(v) = {\tt Ind}\big(G \backslash \{v\}\big)\]
\[\lk_{{\tt Ind}(G)}(v) = {\tt Ind}\big(G \backslash (\{v\} \cup N(v))\big).\]
\end{lemma}

We will need the notion of a folding of a reflexive (loops on all vertices) graph $G$.  If a graph $G$ has vertices $v, w$ such that $N(v) \subseteq N(w)$ then we call the graph homomorphism $G \rightarrow G \backslash \{v\}$ which sends $v \mapsto w$ a \emph{folding}.  A reflexive graph $G$ is called \emph{dismantlable} if there exists a sequences of foldings that results in a single looped vertex (see \cite{Doc} for more information regarding foldings of graphs).  A flag simplicial complex $\Delta = {\tt Cl}(G)$ obtained as the clique complex of some reflexive graph $G$ is called \emph{dismantlable} if the underlying graph $G$ is dismantlable.  One can check that a folding of a graph $G \rightarrow G \backslash \{v\}$ induces an elementary collapse of the clique complexes ${\tt Cl}(G) \searrow {\tt Cl}(G \backslash \{v\})$ which preserves (simple) homotopy type.  Hence if $\Delta$ is a flag simplicial complex we have for any field $k$ the following string of implications. 

\begin{center}
Dismantlable $\Rightarrow$ collapsible $\Rightarrow$ contractible $\Rightarrow$ ${\mathbb Z}$-acyclic $\Rightarrow$ $k$-acyclic.
\end{center}

\noindent
We refer to \cite{B} for details regarding all undefined terms as well as a discussion regarding the chain of implications. 

\subsubsection{Stanley-Reisner rings and edge ideals of graphs}

We next review some notions from commutative algebra and specifically the theory of Stanley-Reisner rings.  For more details and undefined terms we refer to \cite{MS}.  Throughout the paper we will let $\Delta$ denote a simplicial complex on the vertices $[n]$, and will let $S := k[x_1, \dots, x_n]$ denote the polynomial ring on $n$ variables.  The \emph{Stanley-Reisner ideal} of $\Delta$, which we denote $I_{\Delta}$, is by definition the ideal in $S$ generated by all monomials $x_{\sigma}$ corresponding to nonfaces $\sigma \notin \Delta$.  The \emph{Stanley-Reisner ring} of $\Delta$ is by definition $S/I_{\Delta}$, and we will use $R_{\Delta}$ to denote this ring.  One can see that $\dim R_{\Delta}$, the (Krull) dimension of $R_{\Delta}$ is equal to $\dim(\Delta) + 1$.  The ring $R_{\Delta}$ is called \emph{Cohen-Macaulay (CM)} if $\depth R_{\Delta} = \dim R_{\Delta}$.

Suppose we have a minimal free resolution of $R_{\Delta}$ of the form
\[0 \rightarrow \displaystyle{\bigoplus_{j} S[-j]^{\beta_{\ell, j}}} \rightarrow \cdots \rightarrow \displaystyle{\bigoplus_{j} S[-j]^{\beta_{i, j}}} \rightarrow \cdots \rightarrow \displaystyle{\bigoplus_{j} S[-j]^{\beta_{1, j}}} \rightarrow S \rightarrow S/I_\Delta \rightarrow 0 \]
then the numbers $\beta_{i,j}$ are independent of the resolution and are called the (coarsely graded) \emph{Betti numbers} of $R_{\Delta}$ (or of $\Delta$), which we denote $\beta_{i,j}$.  The number $\ell$ (the length of the resolution) is called the \emph{projective dimension} of $\Delta$, which we will denote $\pdim(\Delta)$.  By the Auslander Buchsbaum formula, we have $\dim S - \depth R_{\Delta} = \pdim R_{\Delta}$.  

Note that a resolution of $R_\Delta$ as above can be thought of as a resolution of the ideal $I_\Delta$ (and vice versa) according to
\[0 \rightarrow \displaystyle{\bigoplus_{j} S[-j]^{\beta_{\ell, j}}} \rightarrow \cdots \rightarrow \displaystyle{\bigoplus_{j} S[-j]^{\beta_{i, j}}} \rightarrow \cdots \rightarrow \displaystyle{\bigoplus_{j} S[-j]^{\beta_{1, j}}} \rightarrow \displaystyle{\bigoplus_{j} S[-j]^{\beta_{0, j}}} \rightarrow I \rightarrow 0 \]
where the basis elements of $\displaystyle{\bigoplus_{j} S[-j]^{\beta_{0, j}}}$ correspond to a minimal set of generators of the ideal $I_\Delta$.  Hence we will sometimes not distinguish between resolutions of the Stanley-Reisner ring and the ideal.  We say that $I_{\Delta}$ (or just $\Delta$) has a \emph{$d$-linear} resolution if $\beta_{i,j} = 0$ whenever $j-i \neq d-1$ for all $i \geq 0$.

It turns out that there is a strong connection between the topology of the simplicial complex $\Delta$ and the structure of the resolution of $R_\Delta$.  One of the most useful results for us will be the so-called Hochster's formula (Theorem 5.1, \cite{H1}).

\begin{theorem}[Hochster's formula] \label{hochster}
For $i > 0$ the Betti numbers $\beta_{i,j}$ of a simplicial complex $\Delta$ are given by
\[ \beta_{i,j}(\Delta)=\sum_{W\in {\Delta^0 \choose j}} \dim_k \tilde{H}_{j-i-1} (\Delta[W];k).\]
\end{theorem}

In this paper we will (most often) restrict ourselves to the case $\Delta$ is a \emph{clique complex}, which by definition means the minimal non-faces of $\Delta$ are 1-simplices (edges).  Hence $I_{\Delta}$ is generated in degree 2.  The minimal nonfaces of $\Delta$ can then be considered a graph $G$, and in this case $I_{\Delta}$ is called the \emph{edge ideal} of the graph $G$.  Note that we can recover $\Delta$ as ${\tt Ind}(G)$, the independence complex of $G$, or equivalently as $\Delta(\bar G)$, the clique complex of the complement $\bar G$; we will adopt both perspectives in different parts of this paper.  To simplify notation we will use $I_G := I_{{\tt Ind}(G)}$ (resp. $R_G := R_{{\tt Ind}(G)}$) to denote the Stanley-Reisner ideal (resp. ring) associated to the graph $G$.  The ideal $I_G$ is called the \emph{edge ideal} of $G$.  We will often speak of algebraic properties of a graph $G$ and by this we mean the ring $R_G$ obtained as the quotient of $S$ by the edge ideal $I_G$.

\section{Complements of chordal graphs}
In this section we consider edge ideals $I_G$ in the case that $\bar G$ (the \emph{complement} of $G$) is a chordal graph.  A classical result in this context is a theorem of Fr\"{o}berg (\cite{F}) which states that the edge ideal $I_G$ has a linear resolution if and only if $\bar G$ is chordal.  Our main results in this section include a short proof of this theorem as well as an enumerative interpretation of the relevant Betti numbers.  We then turn to a consideration of bipartite graphs whose complements are chordal; it has been shown by Corso and Nagel (see \cite{CN}) that this class coincides with the so-called Ferrers graphs (see below for a definition).  We recover a formula from \cite{CN} regarding the Betti numbers of Ferrers graphs in terms of the associated Ferrers diagram and also give an enumerative interpretation of these numbers, answering a question raised in \cite{CN}.

Chordal graphs have several characterizations.  Perhaps the most straightforward definition is the following: a graph $G$ is \emph{chordal} if each cycle of length four or more has a \emph{chord}, an edge joining two vertices that are not adjacent in the cycle.  One can show (see \cite{D}) that chordal graphs are obtained recursively by attaching complete graphs to chordal graphs along complete graphs.  Note that this implies that in any chordal graph $G$ there exists a vertex $v \in G$ such that the neighborhood $N(G)$ induces a complete graph (take $v$ to be one of the vertices of $K_n$).

This last condition is often phrased in terms of the clique complex of the graph in the following way.  A facet $F$ of a simplicial complex $\Delta$ is called a \emph{leaf} if there exists a \emph{branch} facet $G \neq F$ such that $H \cap F \subseteq G \cap F$ for all facets $H \neq F$ of $\Delta$.  A simplicial complex $\Delta$ is a \emph{quasi-forest} if there is an ordering of the facets $(F_1, \cdots, F_k)$ such that $F_i$ is a leaf of $<F_1, \cdots F_{k-1}>$.   One can show that quasi-forests are precisely the clique complexes of chordal graphs.

\subsection{Betti numbers and linearity}
Suppose $G$ is the complement of a chordal graph.  As mentioned above, we can think of $I_G$ as the Stanley-Reisner ideal of either ${\tt Ind}(G)$ (the independence complex $G$) or of ${\tt Cl}(\bar G)$, the clique complex of the complement $\bar G$, which is assumed to be chordal.

Our study of the Betti numbers of complements of chordal graphs relies on the following simple observation regarding independence complexes of such graphs.

\begin{lemma}\label{lemma:chordal} If $G$ is a graph such that the complement $\bar{G}$ is a chordal graph with $c$ connected components, then
${\tt Ind}(G) = {\tt Cl}(\bar G)$ is homotopy equivalent to $c$ disjoint points.
\end{lemma}
\begin{proof}
We proceed by induction on the number of vertices of $G$. The lemma is clearly true for the one vertex graph and so we assume that $G$ has more than one vertex.  If there is an isolated vertex $v$ in $\bar{G}$ then ${\tt Cl}(\bar G)$ is homotopy equivalent to the disjoint union of ${\tt Cl}(\bar G \setminus \{v\})$ and a point.  If there are no isolated vertices in $\bar{G}$, we use the fact that any chordal graph has a vertex $v \in G$ whose neighborhood induces a complete graph.  The neighborhood $N(v)$ in $\bar{G}$ is nonempty since $v$ is not isolated by assumption.  For any vertex $w \in N(v)$ we have $N(v) \subseteq N(w)$ and hence ${\tt Cl}(\bar G)$ folds onto the homotopy equivalent ${\tt Cl}(\bar G \setminus \{v\}) = {\tt Ind}(G \setminus \{v\})$.  Removing $v$ in this case did not change the number of connected components of $\bar{G}$.
\end{proof}

This then gives us a formula for the Betti numbers of complements of chordal graphs.

\begin{theorem}\label{compchordalbetti}
Let $\bar{G}$ be a chordal graph. If $i\neq j-1$ then $\beta_{i,j}(G)=0$ and otherwise
\[\beta_{i,j}(G)=\sum_{I\in {V(G) \choose j}} (-1+\# \textrm{ connected components of }\overline{G[I]})\] 
\end{theorem}

\begin{proof}
We employ Hochster's formula (Theorem \ref{hochster}).  Since induced subgraphs of chordal graphs are chordal, Lemma \ref{lemma:chordal} implies that the only nontrivial reduced homology we need to consider is in dimension 0, which in this case is determined by the number of connected components of the induced subgraphs.  The result follows.
\end{proof}

\begin{corollary} Suppose $G$ be a graph with $n$ vertices such that $\bar{G}$ is chordal.  If $\bar G$ is a complete graph then the projective dimension of $G$ is 0, and otherwise the projective dimension is $M-1$, where $M$  is the largest number of vertices in an induced disconnected graph of $\bar G$.
\end{corollary}

In other words, if $\bar G$ is $k$-connected but not $(k+1)$-connected, then the projective dimension of $R_G$ is $n - k- 1$.  Applying the Auslander-Buchsbaum formula we obtain $\dim S - \depth R_G = \pdim R_G$, and from this it follows that the depth of $R_G$ is $k+1$.

As mentioned, we can also give a short proof of the following theorem of Fr\"{o}berg from \cite{F}.

\begin{theorem}\label{Froberg}
For any graph $G$ the edge ideal $I_G$ has a $2$-linear minimal resolution if and only if $G$ is the complement of a chordal graph.
\end{theorem}

\begin{proof}
If $\bar G$ is chordal then Theorem \ref{compchordalbetti} implies that the only nonzero Betti numbers $\beta_{i,j}$ occur when $i = j-1$.  Hence $I_G$ has a 2-linear resolution.  If $\bar G$ is not chordal, there exists an induced cycle $C_j \subseteq \bar G$ of length $j > 3$ and this yields a nonzero element in $\tilde H_1\big({\tt Cl}(C_j)\big) = \tilde H_{j-(j-2)-1}\big({\tt Cl}(C_j)\big)$.  Hochester's formula then implies $\beta_{j-2},j \neq 0$ and hence $I_G$ does not have a 2-linear resolution.
\end{proof}

Among the complements of chordal graphs there are certain graphs that we can easily verify to be Cohen-Macaulay.  For this we need the following notion. 

\begin{definition}
A \emph{$d$-tree} $G$ is a chordal reflexive graph whose clique complex ${\tt Cl}(G)$ is pure of dimension $d+1$, and admits an ordering of the facets $(F_1, \cdots, F_k)$ such that $F_i ~\cap~<F_1, \cdots F_{k-1}>$ is a $d$-simplex.
\end{definition}

Recall that we can identify the edge ideal $I_G$ of a graph $G$ with the Stanley-Reisner ideal of the complex ${\tt Ind}(G) = {\tt Cl}(\bar G)$.  We see that if a graph $H$ is a $d$-tree then then complex ${\tt Cl}(H)$ is pure and shellable.  Purity is part of the definition of a $d$-tree and the ordering of the facets as above determines a shelling order.  As discussed above, we know that a pure shellable complex is Cohen-Macaulay and hence complements of $d$-trees are Cohen-Macaulay.  We record this as a proposition.
\begin{proposition}
Suppose $G$ is a graph such that the complement $\bar G$ is a $d$-tree.  Then the complex ${\tt Ind}(G)$ is pure and shellable, and hence the ring $R_G$ is Cohen Macaulay.
\end{proposition}

\subsection{Ferrers graphs}

In this section we turn our attention to complements of chordal graphs which are also \emph{bipartite}.  It is shown by Corso and Nagel in \cite{CN} that the class of such graphs corresponds to the class of Ferrers graphs, which are defined as follows.  Given a Ferrers diagram (a partition) with row lengths $\lambda_1 \geq \lambda_2 \geq \cdots \geq \lambda_m$, the \emph{Ferrers graph} $G_{\lambda}$ is a bipartite graph with vertex set $\{r_1,r_2, \dots, r_m\} \coprod \{c_1,c_2, \dots,c_{\lambda_1}\}$ and with adjacency given by $r_i \sim c_j$ and edge if $j\leq \lambda_i$. 

In \cite{CN} the authors construct minimal (cellular) resolutions for the edge ideals of Ferrers graphs and give an explicit formula for their Betti numbers.  We wish to apply our basic combinatorial topological tools to understand the independence complex of such graphs; in this way we recover the formula for the Betti numbers and in the process give a simple enumerative interpretation for these numbers in terms of the Ferrers diagram (answering a question posed in \cite{CN}) 

\begin{proposition}\label{Ferrerfold} Suppose $G$ is a Ferrers graph associated to a Ferrers diagram $\lambda = (\lambda_1 \geq \cdots \geq \lambda_n)$. If $\lambda_1=\cdots=\lambda_m$ (so that $G_\lambda$ is a complete bipartite graph) then ${\tt Ind}(G_\lambda)$ is homotopy equivalent to a space of two disjoint points, and otherwise it is contractible.
\end{proposition}

\begin{proof}
The neighborhood of $r_i$ includes the neighborhood of $r_m$ for all $1\leq i<m$, and hence in the complex ${\tt Ind}(G)$ we can fold away the vertices $r_1,r_2, dots,r_{m-1}$.
If $\lambda_1>\lambda_m$ then the vertex $c_{\lambda_1}$ is isolated after the foldings and
thus ${\tt Ind}(G_\lambda)$ is a cone with apex $c_{\lambda_1}$ and hence contractible.  If $\lambda_1=\lambda_m$
then we are left with a star with center $r_m$. We can continue to fold away $c_2,c_3,\ldots,c_{\lambda_1}$
since they have the same neighborhood as $c_1$ and we are left with the two adjacent vertices $r_m$ and $c_1$.  The result follows since the independence complex of an edge is two disjoint points.
\end{proof}

We next turn to our desired combinatorial interpretation of the Betti numbers of the ideals associated to Ferrers graphs.  If $\lambda = (\lambda_1 \geq \dots \geq \lambda_n)$ is a Ferrers diagram we define an \textit{$l \times w$ rectangle} in $\lambda$ to be a choice of $l$ rows $r_{i_1} < r_{i_2} < \cdots < r_{i_l}$ and $w$ columns $c_{j_1} < c_{j_2} < \cdots < c_{j_w}$ such that $\lambda$ contains each of the resulting entries, i.e. $\lambda_{i_l} \geq j_w$.  We say that the rectangle has \textit{size} $l + w$. 

\begin{theorem}\label{Ferrerbetti}
If $G_{\lambda}$ is a Ferrers graph associated to the partition $\lambda = (\lambda_1 \geq \cdots \geq \lambda_n)$, then the Betti numbers of $G_{\lambda}$ are zero unless $j = i+1$, in which case $\beta_{i,i+1}(G_{\lambda})$ is the number of rectangles of size $i+1$ in $\lambda$.  This number is given explictly by:
\[\beta_{i,i+1}(G_{\lambda}) = {\lambda_1 \choose i} + {\lambda_2+1 \choose i} + {\lambda_3+2 \choose i} + \cdots + {\lambda_n + n-1 \choose i} - {n \choose i+1}.\]
\end{theorem}

\begin{proof}
We use Hochester's formula and Proposition \ref{Ferrerfold}.  The subcomplex of ${\tt Ind}(G_{\lambda})$ induced by a choice of $j$ vertices is precisely the independence complex of the subgraph $H$ of $G_{\lambda}$ induced on those vertices.  An induced subgraph of a Ferrers graph is a Ferrers graph and from Proposition \ref{Ferrerfold} we know that the induced complex ${\tt Ind}(H)$ has nonzero reduced homology only if the underlying subgraph $H \subseteq G_\lambda$ is a complete bipartite subgraph, in which case $j = i+1$ and $\dim_k \tilde H_{j-i-1}({\tt Ind}(H);k) = 1$.  An induced complete bipartite graph on $j = i+1$ vertices in $G_\lambda$ corresponds precisely to a choice of an $l \times w$ rectangle with $l + w = j$, where $\{r_{i_1}, \dots, r_{i_l}\}$ and $\{c_{j_1}, \dots, c_{j_w}\}$ are the vertex set.

To determine the formula we follow the strategy employed in \cite{CN}, where the authors use algebraic means to determine the Betti numbers.  Here we proceed with the same inductive strategy but only employ the combinatorial data at hand.  

We use induction on $n$.  If $n = 1$ then $\lambda = \lambda_1$ and the number of rectangles of size $i+1$ is ${\lambda_1 \choose i} = {\lambda_1 \choose i} - {1 \choose i+1}$.

Next we suppose $n \geq 2$ and proceed by induction on $m := \lambda_n$.  Let $\lambda^\prime := (\lambda_1 \geq \lambda_2 \geq \cdots \geq \lambda_{n-1} \geq \lambda_n - 1)$ be the Ferrers diagram obtained by subtracting 1 from the entry $\lambda_n$ in $\lambda$.  First suppose $m = 1$ so that $\lambda^\prime$ has $n-1$ rows.  When we add the $\lambda_n = 1$ entry to the Ferrers diagram $\lambda^\prime$ the only new rectangles of size $i+1$ that we get are $(i+1) \times 1$ rectangles with the entry $\lambda_n$ included.  There are ${n-1 \choose i-1}$ such rectangles, and hence by induction we have
\[ \beta_{i,i+1}(G_\lambda) = \beta_{i,i+1}(G_{\lambda^\prime}) + {n-1 \choose i-1}\]
\[ = {\lambda_1 \choose i} + {\lambda_2+1 \choose i} + \cdots + {\lambda_{n-1} + n - 1-1 \choose i} - {n-1 \choose i+1} + {n-1 \choose i-1}\]
\[ = {\lambda_1 \choose i} + {\lambda_2+1 \choose i} + \cdots + {\lambda_{n-1} + n - 2 \choose i} - {n-1 \choose i+1} + {n \choose i} - {n-1 \choose i}\]
\[ = {\lambda_1 \choose i} + {\lambda_2+1 \choose i} + \cdots + {\lambda_{n-1} + n - 2 \choose i} - {\lambda_n + n - 1 \choose i} - {n \choose i+1}.\]

Now, if $m > 1$ we see that the rectangles of size $i+1$ in $\lambda$ are precisely those in $\lambda^\prime$ along with the rectangles of size $i+1$ in $\lambda$ which include the entry $(n,\lambda_n)$.  The number of rectangles of the latter kind is ${\lambda_n + n -2 \choose i-1}$ since we choose the remaining rows from $\{r_1, \dots, r_{n-1}\}$ and the columns from $\{c_1, \dots, c_{\lambda_n-1}\}$.  Hence by induction on $m$ we get
\[ \beta_{i,i+1}(G_\lambda) = \beta_{i,i+1}(G_{\lambda^\prime}) + {\lambda_n + n - 2 \choose i-1}\]
\[ = {\lambda_1 \choose i} + \cdots + {\lambda_n - 1 + n - 1 \choose i} - {n \choose i+1} + {\lambda_n + n - 2 \choose i-1}\]
\[ = {\lambda_1 \choose i} + \cdots + {\lambda_n + n - 1 \choose i} - {n \choose i+1}.\]

\end{proof}

In particular the edge ideal of a Ferrers graphs has a 2-linear minimal free resolution.  This of course also follows from Fr\"{o}berg's Theorem \ref{Froberg} and the fact (mentioned above) that the complements of Ferrers graphs are chordal.

\section{Chordal graphs, ears and whiskers}
In this section we consider edge ideals $I_G$ in that case that $G$ is a chordal graph.  Perhaps the strongest result in this area is a theorem of Francisco and Van Tuyl from \cite{FT} which says that the ring $R_G$ is sequentially Cohen-Macaulay whenever the graph $G$ is chordal.  We say that a graded $S$-module is \emph{sequentially Cohen-Macaulay} (over $k$) if there exists a finite filtration of graded $S$-modules
\[0 = M_0 \subset M_1 \cdots \subset M_j = M\]
\noindent
such that each quotient $M_i/M_{i-1}$ is Cohen-Macaulay, and such that the (Krull) dimensions of the quotients are increasing:
\[\dim(M_1/M_0) < \dim(M_2/M_1) < \cdots < \dim(M_j/M_{j-1}).\]
Here we present a short proof of the following strengthening of the result from \cite{FT}.

\begin{theorem}\label{chordalVD}
If $G$ is a chordal graph then the complex ${\tt Ind}(G)$ is vertex-decomposable, and hence the associated edge deal $I_G$ is sequentially Cohen-Macaulay.
\end{theorem}

\begin{proof}
We use induction on the number of vertices of $G$.  Since $G$ is chordal there exists a vertex $x$ such that $N(x) = \{v, v_1, \dots v_k\}$ is a complete graph.  By Lemma \ref{indcomplex} we have that $\del_{{\tt Ind}(G)} (v) = {\tt Ind}\big(G \backslash \{v\}\big)$ and $\lk_{{\tt Ind}(G)}(v) = {\tt Ind}\big(G \backslash (\{v\} \cup N(v))\big)$, and hence by induction both complexes are vertex-decomposable.  Also, if $\sigma$ is a maximal face of $\del_{{\tt Ind}(G)} (v)$ then $\sigma$ must contain an element of $\{x, v_1, \dots, v_k\}$, and hence must be a maximal face of of ${\tt Ind}(G)$.  Hence $\Delta = {\tt Ind}(G)$ is vertex decomposable.
\end{proof}

A related result in this area is the main theorem from \cite{BM}, where it is shown that the order complex of a (finite) interval order is shellable.  An interval order is a poset whose elements are given by intervals in the real line, with disjoint intervals ordered according to their relative position.  The order complex of such a poset corresponds to the independence complex of a so-called \emph{interval graph}, a graph whose vertices are given by intervals on the real line with adjacency given by \emph{intersecting} intervals.  One can see that interval graphs are chordal, and hence Theorem \ref{chordalVD} is a strengthening of the main result from \cite{BM}. 

\subsection{Ears and whiskers}

In \cite{FT} the authors identify some non-chordal graphs whose edge ideals are sequentially Cohen-Macaulay; perhaps the easiest example is the 5-cycle.  In addition, a general procedure which we call `adding an ear' is described which the authors suggest (according to some computer experiments) might produce (in general non-chordal) graphs which are sequentially Cohen-Macaulay.  We can use our methods to confirm this (Proposition \ref{earcycle}).  For this we will employ the following lemma, which gives us a general condition to establish when a graph is sequentially Cohen-Macaulay.

\begin{lemma}\label{SCMgraph}
Suppose $G$ is a graph with vertices $u$ and $v$ such that $N(u) \cup \{u\} \subseteq N(v) \cup \{v\}$ and such that the complexes ${\tt Ind}(G \backslash \{v\})$ and ${\tt Ind}\big(G \backslash (\{v\} \cup N(v))\big)$ are both vertex decomposable.  Then the complex $\Delta = {\tt Ind}(G)$ is vertex decomposable and hence $R_G$ is sequentially Cohen-Macaulay.
\end{lemma}

\begin{proof}
We verify the conditions given in Definition \ref{defVD}, with $v$ as our chosen vertex.  According to Lemma \ref{indcomplex} we are left to check that every facet of $\del_\Delta(v) = {\tt Ind}(G \backslash \{v\})$ is a facet of $\Delta$.  Let $\sigma$ be a facet of $\del_\Delta(v)$ and suppose by contradiction that $\sigma \cup \{v\}$ is a facet of $\Delta$.  Then $u \in \sigma$ since $N(u) \subseteq N(v)$.  But $u$ and $v$ are adjacent since $u \in N(v)$, and hence $u$ and $v$ cannot both be elements of $\sigma$.
\end{proof}

We can then use this lemma to prove the following result, first suggested in \cite{FT}.  If $G$ is a graph with some specified edge $e$ then \emph{adding an ear} to $G$ is by definition adding a disjoint 3-cycle to $G$ and identifying one of its edges with $e$ (see Figure 1).

\begin{proposition}\label{earcycle}
For any $r \geq 3$, let $\tilde C_r$ be the graph obtained by adding an ear to the $r$-cycle $C_r$.  Then the complex ${\tt Ind}(\tilde C_r)$ is vertex-decomposable and hence the graph $\tilde C_r$ is sequentially Cohen-Macaulay.
\end{proposition}

\begin{proof}
Take $x$ to be the vertex added to the $r$-cycle and $v$ to be one of its neighbors, and apply Lemma \ref{SCMgraph}.  Note that $G \backslash \{v\}$ and $G \backslash (\{v\} \cup N(v))$ are both chordal graphs and hence the associated independence complexes are vertex decomposable.
\end{proof}

\begin{center}
\epsfig{file=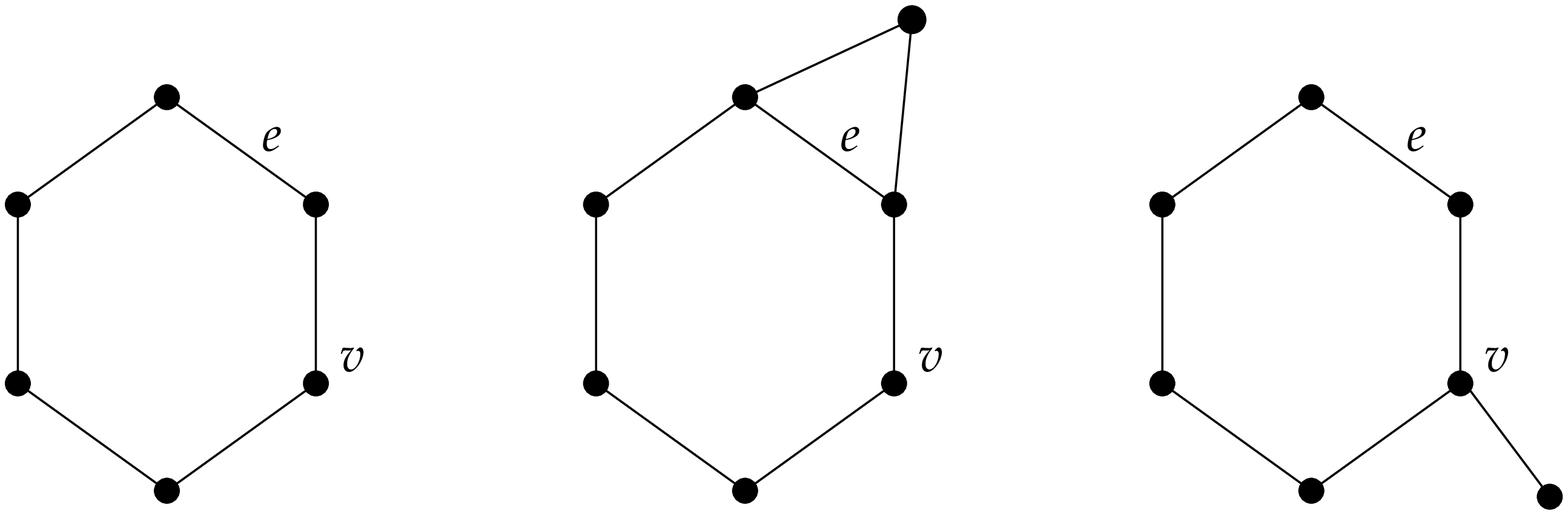, height=1.4 in, width = 4 in}

{Figure 1: Adding an ear at the edge $e$, adding a whisker at the vertex $v$.}
\end{center}

The idea of making small modifications to a graph in order to obtain a (sequentially) Cohen-Macaulay ideal is further explored in other papers.  In \cite{FH} and \cite{V} the authors investigate the notion of `adding a \emph{whisker}' to a vertex $v \in G$, which by definition means adding a new vertex $v^\prime$ and adding a single edge $(v,v^\prime)$; see Figure 1.  The following is a strengthening of one of the theorems of Villareal from \cite{V}  

\begin{theorem}\label{whiskerVD}
Suppose $G$ is a graph and let $G^\prime$ be the graph obtained by adding a whisker at every vertex $v \in G$.  Then the complex ${\tt Ind}(G^\prime)$ is pure and vertex-decomposable and hence the ideal $I_{G^\prime}$ is Cohen-Macaulay.
\end{theorem}

\begin{proof}
For convenience let $\Delta := {\tt Ind}(G^\prime)$.  If $G$ has $n$ vertices then every facet of $\Delta$ has $n$ vertices since in every maximal independent set we choose exactly one vertex from the set $\{v, v^\prime\}$.  To show that $\Delta$ is vertex-decomposable we use induction on $n$.  If $n = 1$ then ${\tt Ind}(G^\prime)$ is a pair of points and hence vertex-decomposable.  For $n > 1$ we choose some vertex $v \in G$ and observe that $\del_{\Delta}(v)$ is a cone over ${\tt Ind}\big((G \backslash \{v\})^\prime \big)$, which is vertex-decomposable by induction.  Similarly, $\lk_{\Delta}(v)$ is a (possibly iterated) cone over ${\tt Ind}\big(G \backslash (\{v\} \cup N(v))\big)$ and hence vertex-decomposable.
\end{proof} 

In \cite{FH}, Francisco and H\`{a} investigate the effect of adding whiskers to graphs in order to obtain \emph{sequentially} Cohen-Macaulay edge ideals.  One of the main results from that paper is the following.

\begin{theorem}[Francisco, H\`{a}]
Let $G$ be a graph and suppose $S \subseteq V(G)$ such that $G \backslash S$ is a chordal graph or a five-cycle.  Then $G \cup W(S)$ is sequentially Cohen-Macaulay.
\end{theorem}

Although we have not been able to find a new proof of this result using our methods, the following other main result from \cite{FH} does fit nicely into our setup.

\begin{theorem}
Let $G$ be a graph and $S \subseteq G$ a subset of vertices.  If $G \backslash S$ is not sequentially Cohen-Macaulay then neither is $G \cup W(S)$.
\end{theorem}

\begin{proof}
According to the combinatorial definition of sequentially CM provided in Section 2.1, a complex $\Delta$ is sequentially CM if and only if the link $\lk_{\Delta}(F)$ is sequentially acyclic for every face $F \in \Delta$.  The `ends' of the whiskers in $G \cup W(S)$ form an independent set and hence determine a face $F$ in $\Delta := {\tt Ind}\big(G \cup W(S)\big)$.  From Lemma \ref{indcomplex} we have that $\lk_\Delta(F) = {\tt Ind}\big((G \cup W(S)\big)$, which is not sequentially acyclic as $G \backslash S$ is assumed not to be sequentially Cohen-Macaulay.
\end{proof}

\section{Projective dimension and max degree}

In this section we determine bounds on the projective dimension of $R_G$ given local
information regarding the graph $G$.  Recall that by Hochster's formula \ref{hochster} the
projective dimension of $R_G$ is the smallest integer $\ell$ such that
\[\dim_k \tilde{H}_{j-i-1}\big({\tt Ind}(G[W])\big)=0\]
for all $\ell < i \leq j$ and subsets $W$ of $V(G)$ with $j$ vertices.  Hence if we know something
about how the topological connectivity of ${\tt Ind}(G[W])$ depends on the size of $W$
we can bound the projective dimension.  Along these lines we have the following theorem.

\begin{theorem}\label{TheoremPDIM}
Let $\Delta$ be a simplicial complex with $n$ vertices, and suppose $a,b$ are real numbers with $a>0$. If
\[ \dim_k \tilde{H}_t(\Delta[W]) = 0 \]
for all integers $t \leq a |W|+b$ and $W\subseteq \Delta^0$ then the
projective dimension of $R_\Delta$ is at most $n(1-a)-b-1$.
\end{theorem}

\begin{proof}
By Hochster's formula it is enough to show that 
\[ \dim_k \tilde{H}_{j-i-1}\big({\tt Ind}(\Delta[W])\big)=0 \]
for all $j$--subsets $W$ of $\Delta^0$ and $i\geq n(1-a)-b-1$.
By assumption we have that $\dim \tilde{H}_{j-i-1}(\Delta[W]) = 0$
for all $j-i-1\leq aj + b$, and since
\[j-i-1\leq j-(n(1-a)-b-1)-1=j-n(1-a)+b \leq j-j(1-a)+b = aj+b \]
we are done.
\end{proof}

We next apply this theorem to obtain information regarding the projective dimension of various classes of graphs for which we have some information regarding the connectivity of the associated independence complexes.

\begin{corollary}\label{enkeltLemma}
Let $G$ be a graph with $n$ vertices and suppose the maximum degree of $G$ is $d\geq 1$.  Then
the projective dimension of $R_G$ is at most $n\left(1-\frac{1}{2d}\right)+\frac{1}{2d}$.
\end{corollary}
\begin{proof}
If $H$ is a graph with $n$ vertices and maximum degree $d$ we have from \cite{AH} and \cite{M1} that 
 \[\dim_k \tilde{H}_t ({\tt Ind}(H)) =0\]
for all  $t\leq \frac{n-1}{2d} -1$.  We then apply Theorem \ref{TheoremPDIM} with $a=\frac{1}{2d}$ and $b=-1-\frac{1}{2d}$.
\end{proof}

In \cite{ST} Szab\'{o} and Tardos showed that the connectivity bounds from \cite{AH} and \cite{M1} 
on independence complexes are optimal.  Their example, the independence complex of
several complete bipartite graphs of the same order, also shows that the bound
on the projective dimension in Corollary~\ref{enkeltLemma} is optimal.  We point out that one can also explicitly calculate the projective dimension of the edge ideals of these graphs by applying the methods outlined below in Section 6.

Recall that a graph is said to be \emph{claw-free} if no vertex has three pairwise nonadjacent neighbors.  Although it may seem like a somewhat artificial property, a graph that is claw-free quite often enjoys some nice properties (see \cite{claw1,claw2}).  For such graphs we can deduce the following property regarding their edge ideals.

\begin{corollary}
Let $G$ be a claw-free graph with $n$ vertices and suppose that the maximum degree of $G$ is $d\geq 1$.  Then
the projective dimension of $R_G$ is at most $n\left(1-\frac{2}{3d+2}\right)+\frac{2}{3d+2}$.
\end{corollary}
\begin{proof}
It $H$ is a graph with $n$ vertices and maximum degree $d$ we have from \cite{E1} that 
 \[\dim_k \tilde{H}_t ({\tt Ind}(H)) =0\]
for all $t\leq \frac{2n-1}{3d+2} -1$.  We then apply Theorem \ref{TheoremPDIM} with $a=\frac{2}{3d+2}$ and $b=-1-\frac{2}{3d+2}$.
\end{proof}

Finite subsets of the $\mathbb{Z}^2$ lattice constitute another class of graphs for which we have good connectivity bounds on the associated independence complexes.  We can then apply our setup to obtain the following.

\begin{corollary}\label{corPlane}
Let $G$ be a finite subgraph of the $\mathbb{Z}^2$ lattice with $n$ vertices.  Then the projective dimension of $R_G$ is at most $\frac{5n}{6}+\frac{1}{2}$.
\end{corollary}

\begin{proof}
From Proposition 4.3 of \cite{E2} we have that the independence complex of a 
finite subgraph of the  $\mathbb{Z}^2$ lattice with $m$ vertices
is $t$--connected for all $t\leq \frac{m}{6}-\frac{3}{2}$.  Hence to get the result we once again employ Theorem \ref{TheoremPDIM} with $a=\frac{1}{6}$ and $b=-\frac{3}{2}$.
\end{proof}

In \cite{E2} the homotopy types of the independence complexes
of disjoint stars with four edges are determined.  One can
use this to show that the constant $ \frac{5}{6}$ in 
Corollary~\ref{corPlane} cannot be decreased to more than $\frac{4}{5}$.

There are more general bounds on the connectivity of independence
complexes, many of them surveyed in \cite{ABM}, but it is not clear to us if they can readily be used to bound the projective dimension of
edge ideals.

We can also apply Theorem \ref{TheoremPDIM} to ideals that are somewhat more general than edge ideals of graphs.  For this we note that an independent set of a graph $G$ is a collection of vertices  with no connected component of size larger than one. The
Stanley-Reisner ideal $I_G$ is generated by the edges of a graph,
or equivalently, by the connected components of size two.
We generalize the edge ideal to the \emph{component ideal}, defined as follows.
\begin{definition}
Let $G$ be a graph with vertex set $[n]$. Then the
\emph{$r$--component ideal} of $G$ is
\[I_{G;r}=< x_{i_1}x_{i_2}\cdots x_{i_r} | i_1<i_2\cdots <i_r\textrm{\emph{ and }} G[\{i_1,i_2,\ldots i_r\}] \textrm{\emph{ is connected}} > \]
\end{definition}
Note that $I_{G;2}$ is the ordinary edge ideal. The
component ideals are Stanley-Reisner ideals of simplicial
complexes that were defined by Szab\'o and Tardos \cite{ST}.
In their notation, the Stanley-Reisner ideal of $\mathcal{K}_{r-1}$
is $I_{G;r}$. Corollary 2.9 of their paper states that:
\begin{lemma}[Szab\'o, Tardos] \label{LemmaST}
Let $t\geq 0$ and $r \geq 1$ be arbitrary integers. If $G$
is a graph with more than $t(d-1+(d+1)/r)$ vertices and with 
maximum degree $d\geq r-1$, then $\mathcal{K}_r(G)$ is
$(t-1)$--connected.
\end{lemma}
Applying this Lemma we obtain another corollary of Theorem~\ref{TheoremPDIM}.
\begin{corollary}\label{mindreEnkeltLemma}
Let $G$ be a graph with $n$ vertices and suppose the maximum degree of $G$ is $d\geq 1$.  Then for $r\geq 2$ the projective dimension of $S/I_{G;r}$ is at most 
\[n\left(1-\frac{1}{d-1+\frac{d+1}{r-1}}\right)+1+\frac{1}{d-1+\frac{d+1}{r-1}}.\]
\end{corollary}
\begin{proof}
We can reformulate Lemma~\ref{LemmaST} as: If $H$ is a graph with $m$ vertices
and maximal degree at most $d$, then for any integer 
\[t\leq \frac{m-1}{d-1+\frac{d+1}{r-1}}-1 \]
the complex $\mathcal{K}_{r-1}(H)$ is $t$--connected. We now use
\[a= \frac{1}{d-1+\frac{d+1}{r-1}} \textrm{ and } b= -1-\frac{1}{d-1+\frac{d+1}{r-1}} \]
and apply Theorem \ref{TheoremPDIM}.
\end{proof}
\noindent
Note that if we take $r=2$ in Corollary~\ref{mindreEnkeltLemma} we do, as expected, recover Corollary~\ref{enkeltLemma}.

The proof of Corollary~\ref{enkeltLemma} and Corollary~\ref{mindreEnkeltLemma} builds on connectivity theorems from 
\cite{AH} and \cite{ST} using ruined triangulations. The method of ruined triangulations is more discrete geometry
than topology, and a natural question to ask is whether it is possible to prove these corollaries directly, without appealing to Hochster's formula.  We have already used the concept of vertex decomposable simplicial complexes several times in
this paper. As was hinted at earlier, if one assumes that the simplicial complex in question is also pure one obtains stronger properties regarding the Stanley-Reisner ring.  For example if $\Delta$ is vertex-decomposable and pure, then it is shellable and pure, and hence also Cohen-Macaulay.  In \cite{H} Hibi showed that the projective dimension of $S/I_\Delta$ is the smallest $k$ such that the $k$--skeleton $\Delta^{\leq k}$ is Cohen-Macaulay.  In \cite{Z2} Ziegler showed that certain skeletons of chessboard complexes are shellable, and we will follow his strategy to show that in fact they are pure vertex decomposable.  With the result of Hibi this leads to another proof of Corollary~\ref{enkeltLemma}.

In the context of independence complexes, Lemma 1.2 of \cite{Z2} states the following.
\begin{lemma}[Ziegler]~\label{zigge}
Let $G$ be a graph with an isolated vertex $v$. If ${\tt Ind}(G\setminus \{v\})^{\leq k}$ is pure vertex
decomposable then ${\tt Ind}(G)^{\leq k+1}$ is pure vertex decomposable.
\end{lemma}

\begin{theorem}\label{vertDa}
If $d$ is not larger than the maximal degree of a graph $G$ with $n$ vertices, and $k$ an integer
less than $n/(2d)$, then ${\tt Ind}(G)^{\leq k}$ is pure vertex decomposable.
\end{theorem}
\begin{proof}
If $d=0$ then ${\tt Ind}(G)$ is a simplex and all of its skeletons are vertex decomposable. Hence we
can assume that $d\geq 1$.  Note that a facet of ${\tt Ind}(G)$ will have at least $\lfloor n/d \rfloor$ vertices, and hence our skeletons will
always be pure. 

The proof is by induction on $n$. If $n=0$ the statement is true because the empty complex is
vertex decomposable.

Next we assume $n>0$.  We fix a vertex $u \in G$ and let $N(u)=\{v_1,v_2,\ldots, v_c\}$; note that $c\leq d$.
The complex ${\tt Ind}(G \setminus (N(u)\cup \{u\}))^{\leq k-1}$ is vertex decomposable 
by induction since 
\[ k-1 \leq \frac{n}{2d}-1 = \frac{n-2d}{2d} \leq \frac{|V(G)\setminus(N(u)\cup \{u\})|}{2d},\]
and hence by Lemma~\ref{zigge}, the complex ${\tt Ind}(G \setminus N(u))^{\leq k}$
is also vertex decomposable.

The next step is to show that the complex ${\tt Ind}(G \setminus \{v_1,v_2,\ldots,v_{c-1}\})^{\leq k}$
is vertex decomposable.  For this we use Definition \ref{defVD} and investigate the link and deletion of $v_c$. The deletion of
$v_c$ is ${\tt Ind}(G \setminus N(u))^{\leq k}$, which is vertex decomposable. The link of $v_c$
is ${\tt Ind}(G \setminus (N(u)\cup N(v_c)) )^{\leq k-1}$ and this is vertex decomposable by
induction since
\[ k-1 \leq \frac{n}{2d}-1= \frac{n-2d}{2d}  \leq \frac{| V(G) \setminus (N(u)\cup N(v_c)) |}{2d}. \]
We conclude that ${\tt Ind}(G \setminus \{v_1,v_2,\ldots,v_{c-1}\})^{\leq k}$ is vertex decomposable.

Now we repeat the step.  Once again we show that ${\tt Ind}(G \setminus \{v_1,v_2,\ldots,v_{c-2}\})^{\leq k}$
is vertex decomposable by considering the link and deletion of $v_{c-1}$. The deletion of $v_{c-1}$ is exactly
the complex we obtained in the last step above, which we concluded was vertex decomposable.  The link of $v_{c-1}$ is
${\tt Ind}(G \setminus ( \{v_1,v_2,\ldots, v_{c-1}\} \cup N(v_{c-1})) )^{\leq k-1}$ and this is vertex decomposable by
induction since
\[ k-1 \leq \frac{n}{2d}-1= \frac{n-2d}{2d}  \leq \frac{| V(G) \setminus 
    ( \{v_1,v_2,\ldots, v_{c-1}\} \cup N(v_{c-1})) |}{2d}. \]
Hence ${\tt Ind}(G \setminus \{v_1,v_2,\ldots,v_{c-2}\})^{\leq k}$
is vertex decomposable.

We continue with this procedure and after $c-2$ steps we conclude that ${\tt Ind}(G)^{\leq k}$ is vertex decomposable.
\end{proof}

We can apply Theorem~\ref{vertDa} to obtain another proof of Corollary~\ref{enkeltLemma}: if the $k$-skeleton of a 
complex on $n$ vertices is Cohen-Macaulay, then by the Auslander-Buchsbaum formula the projective
dimension of its Stanley-Reisner ring is at most $n-k$. 

\section{Generating functions of Betti numbers}

In this section we encode the graded Betti numbers $\beta_{i,j}$ as coefficients of a certain generating function in two variables.  We use combinatorial topology to determine certain relations among the generating functions and use these to derive results regarding graded betti numbers of edge ideals.  The relevant generating function is defined as follows.
\begin{definition}
$ \mathcal{B}(G;x,y)=\sum_{i,j} \beta_{i,j}(G)x^{j-i}y^i$.
\end{definition}
The two variables in $\mathcal{B}(G;x,y)$ correspond to well known algebraic parameters of the edge ideal: the $y$--degree is the projective dimension of $I_G$ (as discussed in the introduction) and the $x$--degree is the \emph{regularity} of $I_G$.  With Hochster's formula we can rewrite the generating function explicitly as
\[ \mathcal{B}(G;x,y)=\sum_{i,j} 
\sum_{W\subseteq V(G) \atop |W|=j} \dim_k \tilde{H}_{j-i-1} ({\tt Ind}(G[W]);k) x^{j-i}y^i.\]
We wish to use $\mathcal{B}(G;x,y)$ to derive certain properties of edge ideals for some classes of graphs.  We first establish a few easy lemmas.
\begin{lemma}\label{lg1}
If $G$ is a graph with an isolated vertex $v$ then
\[\mathcal{B}(G;x,y)=\mathcal{B}(G\setminus \{v\};x,y).\]
\end{lemma}
\begin{proof}
For every $W\subseteq V(G)$ with $v \in W$ we have that ${\tt Ind}(G[W])$ is a cone with
apex $v$ and hence\linebreak $\dim_k \tilde{H}_{j-i-1} ({\tt Ind}(G[W]);k)=0$.
\end{proof}
\begin{lemma}\label{lg2}
If $G$ is a graph with an isolated edge $uv$ then
\[\mathcal{B}(G;x,y)=(1+xy)\mathcal{B}(G\setminus \{u,v\};x,y).\]
\end{lemma}
\begin{proof}
For every $W\subseteq V(G)$ such that exactly one of $\{u, v\}$ is in $W$ we have that ${\tt Ind}(G[W])$ is a cone and hence $\dim_k \tilde{H}_{j-i-1} ({\tt Ind}(G[W]);k)=0$.  If $\{u,v\}\subseteq W\subseteq V(G)$ then ${\tt Ind}(G[W])$
is a suspension of ${\tt Ind}(G[W\setminus \{u,v\}])$ and we have
\[ \begin{array}{rcl} \dim_k \tilde{H}_{j-i-1} ({\tt Ind}(G[W]);k)
  & = & \dim_k \tilde{H}_{j-i-1} ({\tt susp}({\tt Ind}(G[W] \setminus \{u,v\}));k) \\
  & = & \dim_k \tilde{H}_{j-i-1-1} ({\tt Ind}(G[W] \setminus \{u,v\});k) \\
  & = & \dim_k \tilde{H}_{(j-2)-(i-1)-1} ({\tt Ind}(G[W] \setminus \{u,v\});k).
   \end{array}  \]

In the definition of $\mathcal{B}(G;x,y)$ involving Hochster's formula we consider a sum over subsets $W \subseteq V(G)$.  We now split this sum according to the intersection $\{u,v\}\cap W$. If $\{u,v\} \cap W= \emptyset$ the partial sum is of course $\mathcal{B}(G\setminus \{u,v\};x,y)$.  If exactly one of $\{u, v\}$ is in $W$ we have seen that the partial sum is $0$.  If both $\{u,v\}$ are in $W$ then we use the formula from the previous paragraph to obtain the desired term: 
\[ \begin{array}{cl}
  & \displaystyle  \sum_{i,j} \sum_{u,v\in W\subseteq V(G) \atop |W|=j} 
    \dim_k \tilde{H}_{j-i-1} ({\tt Ind}(G[W]);k) x^{j-i}y^i \\
= & \displaystyle  \sum_{i,j}\sum_{u,v\in W\subseteq V(G) \atop |W|=j} 
    \dim_k \tilde{H}_{(j-2)-(i-1)-1} ({\tt Ind}(G[W] \setminus \{u,v\});k) x^{j-i}y^i \\
= & \displaystyle  xy \sum_{i,j}\sum_{ W\subseteq V(G)\setminus \{u,v\} \atop |W|=j-2} 
    \dim_k \tilde{H}_{(j-2)-(i-1)-1} ({\tt Ind}(G[W] \setminus \{u,v\});k) x^{(j-2)-(i-1)}y^{i-1} \\
= & xy\mathcal{B}(G\setminus \{u,v\};x,y).
   \end{array} \]
\end{proof}
\begin{lemma}\label{lg3}
Let $G$ be a graph with a vertex $v$ and a set of vertices $U = \{u_1, \dots, u_m\}$ all different from $v$.  If $N(v)\subseteq N(u)$
for all $u\in U$, then for $\tilde{U}:=U\cup \{v\}$ we have
\[ \mathcal{B}(G;x,y) = \mathcal{B}(G \setminus \{v\} ;x,y) + (1+y)^k(
 \mathcal{B}(G \setminus U ;x,y) - \mathcal{B}(G \setminus \tilde{U} ;x,y)). \]
\end{lemma}
\begin{proof}
We will use the notion of a folding of a graph as defined in Section 2.1.  In this context we have that a vertex of a graph whose neighborhood dominates the
neighborhood of another vertex can be removed without changing the homotopy type of
the independence complex.  Using this we calculate:
\[ 
\begin{array}{rcl}
  \displaystyle  \sum_{v\in W\subseteq V(G) \atop |W|=j, |W\cap U|=l} 
    \dim_k \tilde{H}_{j-i-1} ({\tt Ind}(G[W]);k) &
 =&
  \displaystyle  \sum_{v\in W\subseteq V(G) \atop |W|=j, |W\cap U|=l} 
    \dim_k \tilde{H}_{j-i-1} ({\tt Ind}(G[W\setminus U]);k) \\
 &=&
 \displaystyle  {k \choose l} \sum_{v\in W\subseteq V(G)\setminus U \atop |W|=j-l} 
    \dim_k \tilde{H}_{(j-l)-(i-l)-1} ({\tt Ind}(G[W]);k) \\
 &=&
 \displaystyle  {k \choose l}( \beta_{i-l,j-l}(G\setminus U) - \beta_{i-l,j-l}(G\setminus \tilde{U})).
\end{array}
\]
We then insert this into the relevant generating functions to obtain the following.
\[ \begin{array}{rcl}
  \mathcal{B}(G;x,y)-\mathcal{B}(G\setminus \{v\};x,y) & = & 
  \displaystyle \sum_{i,j} \sum_{v\in W\subseteq V(G) \atop |W|=j} 
    \dim_k \tilde{H}_{j-i-1} ({\tt Ind}(G[W]);k)x^{j-i}y^i \\ 
& = & 
  \displaystyle \sum_{i,j}\sum_{l=0}^k \sum_{v\in W\subseteq V(G) \atop |W|=j, |W\cap U|=l} 
    \dim_k \tilde{H}_{j-i-1} ({\tt Ind}(G[W]);k)x^{j-i}y^i \\
& = & 
  \displaystyle \sum_{i,j}\sum_{l=0}^k 
   {k \choose l}( \beta_{i-l,j-l}(G\setminus U) - \beta_{i-l,j-l}(G\setminus \tilde{U}))
   x^{j-i}y^i \\
& = & 
  \displaystyle \sum_{l=0}^k {k \choose l} y^l \sum_{i,j}
  ( \beta_{i-l,j-l}(G\setminus U) - \beta_{i-l,j-l}(G\setminus \tilde{U}))x^{(j-l)-(i-l)}y^{i-l} \\
& = & 
  \displaystyle \sum_{l=0}^k {k \choose l} y^l (\mathcal{B}(G \setminus U ;x,y) - \mathcal{B}(G \setminus \tilde{U} ;x,y))\\
& = & (1+y)^k (\mathcal{B}(G \setminus U ;x,y) - \mathcal{B}(G \setminus \tilde{U} ;x,y)).
\end{array}
\]
\end{proof}
One special case of Lemma \ref{lg3} is quite useful.
\begin{corollary}\label{cg1}
If $G$ is a graph with a vertex $v$ such that $N(v)=\{w\}$ then
\[ \mathcal{B}(G;x,y) =\mathcal{B}(G\setminus \{v\};x,y)+xy(1+y)^{|N(w)|-1} \mathcal{B}(G\setminus (N(w)\cup \{w\}) ;x,y).\]
\end{corollary}
\begin{proof}
From Lemma~\ref{lg3} we have that $\mathcal{B}(G;x,y)$ equals 
\[ \mathcal{B}(G\setminus \{v\};x,y)+(1+y)^{|N(w)|-1}
 ( \mathcal{B}(G\setminus (N(w)\setminus \{v\}) ;x,y) - \mathcal{B}(G\setminus N(w) ;x,y)).\]
In the graph $G\setminus (N(w)\setminus \{v\})$ the edge $vw$ is isolated and hence by Lemma~\ref{lg2} we have
\[ \mathcal{B}(G\setminus (N(w)\setminus \{v\}) ;x,y) = (1+xy) \mathcal{B}(G\setminus (N(w)\cup \{w\}) ;x,y).\]
If we also remove the vertex $v$ we get a cone with apex $w$ and by Lemma~\ref{lg1},
\[ \mathcal{B}(G\setminus N(w) ;x,y) = \mathcal{B}(G\setminus (N(w)\cup \{w\}) ;x,y).\]
\end{proof}

Corollary~\ref{cg1} is a generalization of the main result of Jacques from \cite{JK_ny}, and also many of the results of Jacques and Katzman from \cite{J_ny}.  These authors used different methods and demanded that at most one vertex from $N(w)$ had more than one neighboor. The
following also generalize results from \cite{J_ny} and \cite{JK_ny}.

\begin{theorem}\label{tg1}
 Let $\mathcal{G}$ be the set of graphs defined by
 \begin{itemize}
   \item[(i)] All cycles and complete graphs are in $\mathcal{G}$.
   \item[(ii)] If $G$ and $H$ are in $\mathcal{G}$ then their disjoint union is in $\mathcal{G}$.
   \item[(iii)] Let $G$ be a graph with vertices $\{u,v\}$ such that $N(v)\subseteq N(u)$. If
                $G\setminus \{u\},  G\setminus \{v\},$ and $ G\setminus \{u,v\}$, are in $\mathcal{G}$ then
                so is $G$.
\end{itemize}
 Then for any $G \in \mathcal{G}$ the Betti numbers of $I_G$ do not depend on the ground field $k$.
\end{theorem}
\begin{proof}
  If $G$ is a cycle or a complete graph then this follows directly from homology results of \cite{K_ny},
  and is also calculated in \cite{J_ny}.

  For the other cases we proceed by induction on the number of vertices of $G$.  From Hochster's formula we see that the Betti numbers of a Stanley-Reisner ring do not depend on the ground field if and only if the the homology of all induced complexes are torsion free. Joins of torsion free complexes
  are torsion free \cite{M_ny}, and since taking the disjoint union of graphs corresponds to taking joins of 
  their independence complexes, we see that graphs created with (ii) satisfy our condition.

  Finally, we apply Lemma~\ref{lg3} to conclude that the Betti numbers of graphs created with (iii) do not depend on the ground field.
\end{proof}

\begin{corollary}\label{cg2}
If $G$ is a forest then the Betti numbers of $G$ do not depend on the ground field.
\end{corollary}
\begin{proof}
 We will show that $G\in \mathcal{G}$ and employ Theorem~\ref{tg1}. If no connected component
 of $G$ has more than two vertices then clearly $G\in \mathcal{G}$. If there is a component
 of $G$ with at least three vertices, we let $v$ be a leaf of that component and let $w$ be a
 vertex of distance two from $v$.  We then use Corollary~\ref{cg1} together with the fact
 that subgraphs of forests are forests.
\end{proof}

We can also use Corollary~\ref{cg1} as in the proof of  Corollary~\ref{cg2} to provide a recursive formula for the regularity and projective dimension of forests.  Suppose $v \in G$ is a leaf vertex of a graph $G$ with $N(v) = \{w\}$.  We use the fact that regularity of $I_G$ is the
$x$--degree of $\mathcal{B}(G;x,y)$, and that the projective dimension is the $y$--degree
 together with 
\[ \mathcal{B}(G;x,y) =\mathcal{B}(G\setminus \{v\};x,y)+xy(1+y)^{|N(w)|-1} \mathcal{B}(G\setminus (N(w)\cup \{w\}) ;x,y)\]
to obtain
 \[ \reg(I_G) = \max \big\{\reg(I_{G\setminus \{v\}}), \reg(I_{G\setminus (N(w)\cup \{w\})})+1 \big\} \]
and
 \[ \pdim(G) = \max \big\{ \pdim({G\setminus \{v\}}), \pdim({G\setminus (N(w)\cup \{w\})})+|N(w)| \big\}.\]

\section{Further remarks}
In this paper we used only basic constructions from combinatorial topology to establish results regarding Betti numbers, linearity of resolutions, and (sequential) Cohen-Macaulay properties of edge ideals.  It is our hope that more sophisticated tools from combinatorial topology will have further applications to the study of edge ideals of graphs (and more generally Stanley-Reisner ideals).  Further analysis of the combinatorial properties of certain classes of simplicial complexes can give good candidates for desired algebraic properties of the associated Stanley-Reisner ring (e.g. those that satisfy the conditions in Lemma \ref{SCMgraph}).  In this vein, tools from combinatorial topology may also offer insight into the less well understand class of edge ideals of \emph{hypergraphs} (Stanley-Reisner rings generated in some fixed degree $d > 2$).  At the same time one can ask the question if theorems from the study of Stanley Reisner rings can have applications to the more combinatorial topological study of certain classes of simplicial complexes.  For example the algebraic proof of the theorem from \cite{FH} regarding adding whiskers to chordal graphs gives some combinatorial topological (sequential Cohen-Macaulay) properties of the independence complex of such graphs.  In any case we see potential for interaction between the two fields and hope that this paper leads to further dialogue between mathematicians working in both areas.

\end{document}